\documentclass{article}

\usepackage{amsmath,amsfonts,amssymb,amsthm}
\usepackage{epsfig}
\usepackage{mathrsfs}
\usepackage{enumerate}
\newtheorem{theorem}{Theorem}

\newcommand{\e}{{\rm e}}
\newtheorem{thm}{Theorem}[section]

\newtheorem{prop}[thm]{Proposition}

\newcommand{\Nor}{{\mathrm{Nor}}\,}

\newcommand{\R}{{\mathbb R}}

\newcommand{\E}{{\mathbb E}}
\newcommand{\N}{{\mathbb N}}
\newcommand{\Q}{{\mathbb Q}}

\catcode`@=11
\def\section{%
\setcounter{equation}{0} \setcounter{theorem}{0} \@startsection
{section}{1}{\z@}{-4.0ex plus -1ex minus
    -.2ex}{2.3ex plus .2ex}{\bf}}
\catcode`@=12 \theoremstyle{definition}

\begin{document}

\title{Determination of Boolean models by mean values of mixed volumes}

\author{Daniel Hug and Wolfgang Weil\\
 \\ {\it Department of Mathematics,
 Karlsruhe Institute of Technology}\\
 {\it 76128 Karlsruhe, Germany}}

\date{December 21, 2017}

\maketitle

\begin{abstract}
\noindent
In Weil (2001) formulas were proved for stationary Boolean models $Z$ in $\R^d$ with convex or polyconvex grains, which express the densities of mixed volumes of $Z$ in terms of related mean values of the underlying Poisson particle process $X$. These formulas were then used to show that in dimensions 2 and 3 the mean values of $Z$ determine the intensity $\gamma$ of $X$. For $d=4$ a corresponding result was also stated, but the proof given was incomplete, since in the formula for the mean Euler characteristic $\overline V_0 (Z)$ a term $\overline V^{(0)}_{2,2}(X,X)$ was missing. This was pointed out in Goodey and Weil (2002), where it was also explained that a new decomposition result for mixed volumes and mixed translative functionals would be needed to complete the proof.

Such a general decomposition result is now available based on flag measures of the convex bodies involved (Hug, Rataj and Weil (2013, 2017)). Here, we show that such flag representations not only lead to a correct derivation of the 4-dimensional result but even yield a corresponding uniqueness theorem in all dimensions. In the proof of the latter, we make use of Alesker's representation theorem for translation invariant valuations. We also discuss which shape information can be obtained in this way and comment on the situation in the non-stationary case.

\medskip


\medskip

\end{abstract}

\section{Introduction}

Let ${\cal K}^d$ be the space of  convex bodies (non-empty compact convex sets) in $\R^d, d\ge 2,$ supplied with the Hausdorff metric. A {\it Boolean model with convex grains} is a random closed set $Z\subset \R^d$, generated as the union set of a Poisson particle process $X$ on ${\cal K}^d$,
$$
Z = \bigcup_{K\in X} K .
$$
We refer to \cite{SW}, for background information on random sets, point processes, Boolean models and all further notions and results from Stochastic Geometry which are used in the following without detailed explanation. If $X$ and $Z$ are {\it stationary} (i.e., their distribution is invariant under translations in $\R^d$), the underlying process $X$ is determined (in distribution) by two quantities, the {\it intensity} $\gamma$ (assumed to be positive and finite) and the distribution $\mathbb Q$ of the {\it typical grain}, a probability measure on the subset ${\cal K}^d_0\subset {\cal K}^d$ of convex bodies with Steiner point at the origin. A major problem in applications, for example when a Boolean model is fitted to given data, is to estimate $\gamma$, the mean number of particles per unit volume, from measurements of the union set $Z$.

If $X$ and $Z$ are, in addition, {\it isotropic} (i.e., their distribution is also invariant under rotations), the classical formulas of Miles and Davy (see \cite[Theorem 9.1.4]{SW}) allow such an estimation. The formulas express the mean values $\overline V_j(Z)$ of the (additively extended) intrinsic volumes $V_j, j=0\dots ,d,$ of the Boolean model $Z$ as a triangular array of the mean values
$$
 \overline V_j(X) := \gamma \int_{ {\cal K}^d_0} V_j(K) \,{\mathbb Q}(dK)
$$
of $X$ and read
\begin{align*}
\overline V_d(Z) &= 1-\e^{-\overline V_d(X)} ,
  \\
  \overline V_{d-1}(Z) &= \e^{-\overline
    V_d(X)}\overline V_{d-1}(X) ,
\end{align*}
and
\begin{align}\label{MilesDavy}
  \overline V_j(Z) &=\e^{-\overline V_d(X)}
  \sum_{\mathbf{m}\in\operatorname{mix}(j)}\frac{(-1)^{|{\mathbf{m}|}-1}}{|{\mathbf{m}|}!}
  c_j^d \prod_{i=1}^{|\mathbf{m}|} c_d^{m_i}\overline V_{m_i}(X),
\end{align}
for $j=0,\dots ,d-2$. Here, for $j\in \{0,\ldots,d-1\}$ we define
$$
\operatorname{mix}(j):=\bigcup_{k=1}^{d-j}\operatorname{mix}(j,k),
$$
where $\operatorname{mix}(j,1):=\{(j)\}$ and
$$
\operatorname{mix}(j,k):=\{(m_1,\ldots,m_k)\in\{j+1,\ldots,d-1\}^k:m_1+\ldots+m_k=(k-1)d+j\}
$$
for $k\in\{2,\ldots,d-j\}$. Moreover, we put $|\mathbf{m}|:=k$ if $\mathbf{m}\in \operatorname{mix}(j,k)$. The
constants involved are given by $c^m_j:=m!\kappa_m/(j!\kappa_j)$, where $\kappa_m=\pi^{m/2}/\Gamma(1+m/2)$ is the
volume of an $m$-dimensional unit ball.

This system of equations can be inverted from top to bottom to yield $\gamma = \overline V_0(X)$ in terms of the mean values $\overline V_j(Z)$ for $j=0,\ldots,d$.

Without the condition of isotropy, the estimation of the intensity remains an interesting and challenging problem, some methods are discussed in Section 9.5 of \cite{SW}. One rather effective method, the so-called method of moments, was discussed in a number of papers for the dimensions 2 and 3 (see the references given in \cite{SW}). It uses mean values of direction dependent functionals like the support function and the area measures (both additively extended to polyconvex sets). In \cite{W01} a unified approach to these local results was given by considering mixed volumes and mixed functionals of translative integral geometry instead of (just) intrinsic volumes. The corresponding mean value formulas are
\begin{align*}
  \overline V_d(Z) &= 1-\e^{-\overline V_d(X)} ,
  \\
  \overline V(Z [d-1], K[1]) &= \e^{-\overline
    V_d(X)}\overline V(X [d-1],K [1]) ,
\end{align*}
and
\begin{align}\label{nonisotropic}
  &\binom{d}{j}\overline V(Z [j], K[d-j] ) \nonumber\\
	&\qquad\qquad =\e^{-\overline V_d(X)}
  \sum_{\mathbf{m}\in\operatorname{mix}(j)}
  \frac{(-1)^{{|\mathbf{m}|}-1}}{|\mathbf{m}|!}\overline{V}_{\mathbf{m},d-j}(X,\ldots,X,K^*) ,
  %\overline V_{\mathbf{m}}(X,\dots,X, K[d-j]),
\end{align}
for $j=0,\dots ,d-2$, and all $K\in{\cal K}^d$.
Here, $K^*:=-K$ is the reflection of $K$ in the origin,
$\mathbf{m}=(m_1,\ldots,m_k)\in \operatorname{mix}(j,k)$ and $(\mathbf{m},d-j)\equiv (m_1,\ldots,m_k,d-j)\in
\operatorname{mix}(0,k+1)$, for $k=1, \ldots,d-j$. On the left side of \eqref{nonisotropic}, mean values of mixed volumes of the Boolean model $Z$
are used, whereas on the right, mean values of mixed functionals $V_{{\bf m},d-j}$ of the underlying particle process $X$ are involved. These mixed functionals
  arise from the (iterated) translative integral geometric formula for the intrinsic volumes (especially, from the Euler characteristic $V_0$).
Note that here we simply write $ {V}_{\mathbf{m},d-j}$ and $\overline{V}_{\mathbf{m},d-j}$ instead of $ {V}^{(0)}_{\mathbf{m},d-j}$  and $\overline{V}^{(0)}_{\mathbf{m},d-j}$,  since the upper index is determined by the information provided in the lower index. Furthermore, in the special case where $|\mathbf{m}|=1$, the mixed functionals  are related to
the mixed volumes by \eqref{eqrel}, which implies for the corresponding mean values that
$$
\binom{d}{j}\overline V(X[j], K[d-j])=\overline{V}_{j,d-j}(X,K^*),
$$
for all $K\in{\cal K}^d$.
The question arises whether the knowledge of the mean values $\overline V(Z [j], K[d-j] )$, for $j=0,\dots ,d$, and all $K\in{\cal K}^d$, determines the intensity $\gamma$ of $X$.

Notice that in the marginal cases $j=d$ and $j=0$ we have
 $\overline V(Z [d], K[0] ) = \overline V_d(Z)$ and $\overline V(Z [0], K[d] ) = \overline V_0(Z)V_d(K)$, respectively. Also
$$
\overline V_{\mathbf{m},d}(X,\dots,X, K^*)=\overline{V}_{\mathbf{m}}(X,\ldots,X)V_d(K)
$$
due to the decomposition property of mixed functionals.
If we divide by $V_d(K)$ and make the summation explicit, the equation for $j=0$ reads
\begin{align}\label{Euler}
  &\overline V_0(Z) =\e^{-\overline V_d(X)}\Bigg( \overline V_0(X) -
  \sum_{k=2}^d
  \frac{(-1)^{k}}{k!}\nonumber \\
	&\qquad\qquad\qquad\qquad\qquad \times\sum_{\substack{m_1,\ldots ,m_k=1\\ m_1+\ldots +m_k=
	(k-1)d}}^{d-1}\overline V_{m_1,\dots, m_k}(X,\dots,X)\Bigg).
\end{align}
Thus, in order to determine the intensity $\gamma = \overline V_0(X)$ from this equation, the
 mixed densities $\overline V_{m_1,\dots, m_k}(X,\dots,X)$ have to be obtained, for all indices $m_1,\ldots ,m_k$, by the other equations in \eqref{nonisotropic} for $j=1,\dots ,d$. As was shown in \cite{W01}, this works for dimensions $d=2$ and $d=3$. The main point is that, in these small dimensions, only mixed volumes occur which are of the form
\begin{equation}\label{d-1-case}
V(K[1],M[d-1]) = \frac{1}{d} \int_{S^{d-1}} h^*(K,u)\,S_{d-1}(M,du),
\end{equation}
for $ K,M\in{\cal K}^d$. Here $h^\ast(K,\cdot)$ is the centered support function of $K$ (the support function of the translate of $K$ with Steiner point at the origin) and $S_{d-1}(M,\cdot)$ is a Borel measure on the unit sphere $S^{d-1}$,  the $(d-1)$st area measure of $M$. Moreover, the value of $V(K[1],M[d-1])$ for fixed $K$ and all $M\in{\cal K}^d$ determines $h^*(K,\cdot)$, and for fixed $M$ and all $K\in{\cal K}^d$ it determines $S_{d-1}(M,\cdot)$. Based on these facts, one can show that all mean values $\overline V_{m_1,\dots, m_k}(X,\dots,X)$  in \eqref{Euler} are determined by the higher order mean values of $Z$, as long as the indices $m_i$ are either 1 or $d-1$. This is sufficient in dimensions $d=2$ and $d=3$, but for $d=4$ the formula for the Euler characteristic reads
\begin{align}\label{Euler-4}
\overline V_0(Z) =\e^{-\overline V_4(X)}\Bigl(& \overline V_0(X) -\overline V_{1,3}(X,X) -\frac{1}{2}\overline V_{2,2}(X,X)\\ & + \frac{1}{2}\overline V_{2,3,3}(X,X,X) -\frac{1}{24}\overline V_{3,3,3,3}(X,X,X,X)\Bigr).\nonumber
\end{align}
In \cite[Eq. (20)]{W01} this formula was stated in an incorrect way. Not only were the constants missing, also the term $\overline V_{2,2}(X,X)$ was left out. As was shown there, the mixed expressions $\overline V_{1,3}(X,X)$ and $\overline V_{3,3,3,3}(X,X,X,X)$ are determined by the above-mentioned principle. Moreover, also the term $\overline V_{2,3,3}(X,X,X)$ can be obtained, since it can be expressed as a mixed functional with respect to the Blaschke body $B(X)$ of $X$ (see \cite[Proposition 3]{W01}). However, the missing summand $\overline V_{2,2}(X,X)$ cannot be treated in this way and therefore the proof in \cite{W01} is incomplete.

In fact, as was pointed out in \cite{GW02}, for a corresponding proof in the 4-dimensional case an integral representation similar to \eqref{d-1-case} for the mixed volume $V(K[2],M[2])$ would be necessary. Such a result has been shown now for arbitrary $d$ in \cite{HRW}, namely, for $j\in \{1,\ldots,d-1\}$ we have
\begin{align}\label{j-case}
V&(K[j],M[d-j]) \\
&= \int_{F(d,j+1)} \int_{F(d,d-j+1)} f_j(u,L,u',L')\,\psi_{d-j}(M,d(u',L'))\,\psi_j(K,d(u,L))\nonumber
\end{align}
for all $K,M\in {\cal K}^d$ which are in general relative position. Here $\psi_j(K,\cdot)$ denotes the $j$th {\it flag measure} of $K$, a finite non-negative Borel measure on the space $F(d,j+1)$ of all flags $(u,L)$, where $L$ is an element of the linear Grassmannian $G(d,j+1)$ and $u$ is a unit vector in $L$ (and similarly $\psi_{d-j}(M,\cdot)$ is the corresponding flag measure of $M$ on $F(d,d-j+1)$). The function $f_j$ on $F(d,j+1)\times F(d,d-j+1)$ is a signed measurable function independent of $K$ and $M$. Recently \cite{HRW2}, this result was extended to flag representations of general mixed volumes $V(K_1[n_1],\dots ,K_k[n_k])$ of bodies $K_i\in{\cal K}^d$ in general relative position with natural numbers $n_i$ such that $n_1+\ldots +n_k=d$ and also to flag formulas for mixed translative functionals $V_{m_1,\ldots, m_k}(K_1,\dots ,K_k)$, where $m_1+\ldots +m_k= (k-1)d$.

Using these flag representations and related approximation results (used to avoid any assumption on the relative
position of the involved convex bodies), we can now complete the proof of the 4-dimensional situation, but we will also obtain the following general result.

\begin{theorem}\label{1} Let $Z$ be a stationary Boolean model in $\R^d, d\ge 2,$ with convex grains and
satisfying \eqref{moment}. If for $j=0,\dots ,d$ and all $K\in{\cal K}^d$  the densities of the mixed volumes $\overline V(Z[j], K[d-j])$ are given, then the intensity $\gamma$ of the underlying Poisson particle process $X$ and the mean flag measures
$$
\int_{{\cal K}^d_0} \psi_j(K,\cdot)\, {\mathbb Q} (dK),\quad j=1,\dots ,d-1,
$$
of the particles are determined.
\end{theorem}

Another essential ingredient of the proof of the theorem (necessary already in the 4-dimensional case) is a deep result from valuation theory, namely Alesker's confirmation (see \cite{A1}) of a conjecture of McMullen. This result states (in generalized form) that every translation invariant continuous valuation $\varphi$ on ${\cal K}^d$ which is homogeneous of degree $j\in\{1,\dots ,d-1\}$ can be approximated by special linear combinations of mixed volumes, that is, for each $n\in\N$ there exist $\lambda^n_i\in\R$ and $K_i^n\in {\cal K}^d$,  $i\in\{1,\ldots,n\}$, such that
$$
\varphi (M) = \lim_{n\to\infty} \sum_{i=1}^n \lambda_i^{n} V(M[j],K_i^n[d-j]) ,
$$
for all $ M\in{\cal K}^d$. In addition, the convergence is uniform for convex bodies $M$ contained in the unit ball $B^d$ (see \cite[p.~245--247]{A1}).

In the next section, we collect some background information on the notions and results from convex geometry, valuation theory and stochastic geometry which are used later on. In Section 3, we complete the proof in the 4-dimensional situation as a motivation for the proof of Theorem \ref{1}, which will be given in Section 4. In Section 5 we collect some remarks on extensions and applications of the main result. In particular, we discuss which shape information can be obtained in the non-isotropic case. In the final Section 6, we consider the situation for non-stationary Boolean models where a corresponding uniqueness result remains open for $d\ge 4$. For dimensions $d=2$ and $d=3$, we explain conditions under which a positive answer can be given, modifying thus corresponding statements in \cite{W00} and \cite{GW02}. We also use the opportunity, in Sections 3 and 6, to correct some further misprints in \cite{W01}.

\section{Background information}
We work in Euclidean space $\R^d$, with scalar product $\langle\cdot\,,\cdot\rangle$ and induced norm $\|\cdot\|$. Let
$B^d$ denote the unit ball and $S^{d-1}$ its boundary, the unit sphere. A convex body $K\subset\R^d$ is
a nonempty compact convex set, $\mathcal{K}^d$ denotes the space of convex bodies together with the Hausdorff metric.
We write $\mathcal{K}^d_0$ for the set of all $K\in \mathcal{K}^d$ whose Steiner point is the origin; see \cite{S} and \cite[Chapter 14]{SW} for further details. On $\mathcal{K}^d$ various geometric functionals are defined. The intrinsic
volumes $V_j:\mathcal{K}^d\to[0,\infty)$, for $j=0,\ldots,d$, are positively homogeneous of degree $j$, motion invariant, continuous and additive functionals, which can be introduced via the Steiner formula for the volume $V_d(K+\lambda B^d)$ of the parallel body $K+\lambda B^d, \lambda \ge 0, K\in{\cal K}^d$. Since additivity (the valuation property) is a fundamental property which is crucial for many relations used in the following, we recall that a functional
$\varphi: \mathcal{K}^d\to\R$ is additive (a valuation) if
$$
\varphi(K\cup M)+\varphi(K\cap M)=\varphi(K)+\varphi(M)
$$
for all $K,M\in \mathcal{K}^d$ for which $K\cup M\in \mathcal{K}^d$. The intrinsic volumes can also be viewed as particular examples of mixed volumes of convex bodies. Mixed volumes are continuous functionals $V:(\mathcal{K}^d)^d\to[0,\infty)$ which arise as coefficients in an expansion
$$
V_d\left(\sum_{i=1}^m\lambda_iK_i\right)=\sum_{i_1,\ldots,i_d=1}^m\lambda_{i_1}\cdots\lambda_{i_d}V(K_{i_1},\ldots,K_{i_d})
$$
for $\lambda_i\ge 0$ and $K_i\in\mathcal{K}^d$, $i=1,\ldots,m$. The mixed volume is symmetric in its $d$ arguments and hence uniquely determined by this expansion. Therefore, writing briefly $V(K[j],M[d-j])$ for $V(K,\ldots,K,M,\ldots,M)$ with $j$ copies of $K$ and $d-j$ copies of $M$, we obtain
$$
V_d(K+\lambda M)=\sum_{j=0}^d\lambda^{d-j}\binom{d}{j}V(K[j],M[d-j])
$$
for $\lambda\ge 0$ and $K,M\in\mathcal{K}^d$. In the special case $j=1$ (and similarly for $j=d-1$), as pointed out before
we have the integral representation
$$
V(K[1],M[d-1])=\frac{1}{d}\int_{S^{d-1}}h(K,u)\, S_{d-1}(M,du),
$$
where $h(K,\cdot)$ is the support function of $K$ (considered as a function on $S^{d-1}$) and $S_{d-1}(M,\cdot)$
is a finite Borel measure on the unit sphere (the top order area measure of $M$). Since $S_{d-1}(M,\cdot)$
is centred (that is to say, the centre of the mass distribution is the origin), the integrand is only determined up to a linear function, that is, up to a translate of $K$. Hence this translate can be chosen such that the Steiner point is the origin (see
\eqref{d-1-case}).

An important result  for area measures is Minkowski's existence and uniqueness result, which states that, for any finite,
centred and non-degenerate (not concentrated on a great subsphere) Borel measure $\mu$ on the unit sphere, there is up to a translation a unique convex body
$B_\mu$ such that $\mu=S_{d-1}(B_\mu,\cdot)$.

It is well known that the area measure $S_{d-1}(K,\cdot)$ of a convex body $K\in\mathcal{K}^d$ is one member of the family of area measures $S_j(K,\cdot)$, $j=0,\ldots,d-1$,
which are all finite, centred Borel measures on $S^{d-1}$ and can be obtained via a {\em local} Steiner formula. In fact, they are all special instances of the mixed area measures (see \cite[Section 14.3, p.~611]{SW}). Here, however, we need another extension of the area measures that has recently been studied more intensively. First, we use the renormalization $\Psi_{d-1}(K,\cdot):=\frac{1}{2}S_{d-1}(K,\cdot)$ whose total measure equals $V_{d-1}(K)$. For $k\in\{1,\ldots,d\}$, we then consider the flag space
$$F(d,k):=\{(u,U)\in S^{d-1}\times G(d,k):u\in U\}$$
and on the Borel sets of $F(d,j+1)$, for $j=0,\ldots,d-1$, the flag measure
\begin{equation}\label{flag}
\psi_j(K,\cdot):=\int_{G(d,j+1)}\int_{S^{d-1}\cap U}\mathbf{1}\{(u,U)\in\cdot\}\, \Psi_j^U(K|U,du)\, \nu_{j+1}(dU),
\end{equation}
where $\Psi_j^U(K|U,\cdot)$ is the renormalized area measure of the projection of $K$ to $U$,
with respect to the subspace $U$, and $\nu_{j+1}$ is the rotation invariant Haar probability measure on $G(d,j+1)$.
In particular, the map $K\mapsto \psi_j(K,\cdot)$ is additive and positively homogeneous of degree $j$.
We refer to \cite{HTW} and \cite{Weil17b} for background information on flag measures of convex bodies (in particular, for different normalizations and isomorphic versions), and to
\cite{HRW, HRW2} for integral representations of mixed volumes and mixed functionals with respect to flag measures as needed here
and described in the following sections.

In addition to the mixed volumes, we need certain mixed functionals of translative integral geometry and their
relation to mixed volumes. The mixed functionals we will make use of are determined by the
iterated translative integral formula for the intrisic volumes, that is,
\begin{align*}
&\int_{\R^d}\dots\int_{\R^d}V_j(K_1\cap (K_2+x_2)\cap\ldots\cap (K_k+x_k))\, \lambda_d(dx_2)\cdots \lambda_d(dx_k)\\
&\qquad =\sum_{\substack{m_1,\ldots,m_k=j\\m_1+\ldots+m_k=(k-1)d+j}}^{d}V_{m_1,\ldots,m_k}(K_1,\ldots,K_k),
\end{align*}
for $j\in \{0,\ldots,d\}$, $k\ge 2$ and $K_1,\ldots,K_k\in\mathcal{K}^d$. Here $\lambda_d$ denotes $d$-dimensional Lebesgue measure.
We refer to \cite{S, SW} and to \cite{SchuW} for properties of the mixed functionals $V_{m_1,\ldots,m_k}$ and
various (local and abstract) extensions. As mentioned in the Introduction, mixed volumes and the mixed functionals
of translative integral geometry are different objects, but they are closely related for $k=2$ and $j=0$ where
we have
\begin{equation}\label{eqrel}
\binom{d}{m}V(K[m],M[{d-m}])=V_{m,d-m}(K,M^*).
\end{equation}
The mixed functionals again satisfy translative integral geometric formulas. For instance, we have
\begin{align*}
&\int_{\R^d}V_{m_1,\ldots,m_{k-2},m}(K_1,\ldots,K_{k-2},K_{k-1}\cap (K_k+x))\,\lambda_d(dx)\\
&\qquad=\sum_{\substack{m_{k-1},m_k=m\\m_{k-1}+m_k=d+m}}^{d}
V_{m_1,\ldots,m_{k-2},m_{k-1},m_k}(K_1,\ldots,K_{k-2},K_{k-1},K_k),
\end{align*}
for $k\ge 3$ and $m_1+\ldots+m_{k-2}+m=(k-2)d+j$.

These functionals on convex bodies and their additive extensions to polyconvex sets are used
in the study of Boolean models.
 In the Introduction, a stationary Boolean model was defined as the union set of a stationary
Poisson particle process $X$ in $\mathcal{K}^d$. An alternative description in terms of an
independently marked Poisson point process is given in \cite{SW} (see also the extensive literature cited there).
Writing  $\mathbb{E}$ for expectation with respect to an underlying probability measure $\mathbb{P}$, we
assume in the following
that the intensity measure $\mathbb{E} X$ of $X$ is locally finite and non-zero, and hence by stationarity has the form
$$
\mathbb{E} X(\cdot)=\gamma\int_{\mathcal{K}^d_0}\int_{\R^d}\mathbf{1}\{K+x\in\cdot\}\,\lambda_d(dx)\,\mathbb{Q}(dK),
$$
where $\gamma\in (0,\infty)$ is the intensity of $X$ and $\mathbb{Q}$ is a probability measure on the space of
convex bodies, concentrated on bodies with Steiner point at the origin. The assumption of local finiteness of
the intensity measure is equivalent to
\begin{equation}\label{locfinBM}
\int_{\mathcal{K}^d_0}V_j(K)\, \mathbb{Q}(dK)<\infty,\qquad j=1,\ldots,d,
\end{equation}
and ensures that indeed $Z$ is a random closed set. This assumption is always tacitly made. For our analysis, we
will need the additional  moment assumption
\begin{equation}\label{moment}
\int_{\mathcal{K}^d_0}V_1(K)^{d-2}\, \mathbb{Q}(dK)<\infty ,
\end{equation}
which implies \eqref{locfinBM} for $j=1,\ldots,d-2$, by a special case of \cite[(14.31)]{SW}
and H\"older's inequality. In particular, all moments with exponent $j<d-2$ are finite as well.
If $\varphi$ is a measurable, nonnegative functional on $(\mathcal{K}^d)^{q+r}$ which is
translation invariant in each argument, we define
the mean values (densities) of $\varphi$ for the particle process $X$ by
\begin{align*}
\overline\varphi(X,\ldots,X,M_1,&\ldots ,M_r)\\
&:=\gamma^q\int_{(\mathcal{K}^d_0)^{q}}\varphi(K_1,\ldots,K_q,M_1,\ldots, M_r)\,
\mathbb{Q}^q(d(K_1,\ldots,K_q)),
\end{align*}
for all $M_1,\ldots , M_r\in\mathcal{K}^d$.  Here we allow $r=0$ which gives us $\overline\varphi(X,\ldots,X)$. Thus, in particular, the mean values $\overline V_j(X)$, $\overline V_{\mathbf{m}}
(X,\ldots,X)$ and  $\overline V_{\mathbf{m},d-j}(X,\ldots,X,M)$ from the Introduction are defined and finite,
as a consequence of \eqref{locfinBM} and the translative integral formula. In addition to scalar valued
mean values, the mean values of support functions and area measures are also required. They are defined in the obvious
way by
$$
\overline{h}(X,\cdot):=\gamma\int_{\mathcal{K}^d_0} h(K,\cdot)\, \mathbb{Q}(dK)
$$
and
$$
\overline{S}_{d-1}(X,\cdot):=\gamma\int_{\mathcal{K}^d_0} S_{d-1}(K,\cdot)\, \mathbb{Q}(dK).
$$
It follows from \eqref{locfinBM} that $\overline{h}(X,\cdot)$ is a finite function, which is again a support function of a convex body $M(X)\in {\cal K}^d_0$, the mean body of $X$,
and $\overline{S}_{d-1}(X,\cdot)$ is again a finite, centred Borel measure on the unit sphere. If it is non-degenerate,
then it is the area measure of a convex body $B(X)\in {\cal K}^d_0$, the Blaschke body of $X$,
$$
\overline{S}_{d-1}(X,\cdot)=S_{d-1}(B(X),\cdot).
$$
It is important to keep in mind, that these mean values and $B(X)$ are deterministic objects and $X$ simply reminds of
the dependence on $X$, that is, on $\gamma$ and $\mathbb{Q}$, which are uniquely determined by $X$.

The Blaschke body $B(X)$ of $X$ exists, if the measure $\overline{S}_{d-1}(X,\cdot)$ is not supported by a subsphere of $S^{d-1}$. This means that the distribution ${\mathbb Q}$ is not concentrated on convex bodies which all lie in parallel hyperplanes. We can and will assume this without loss of generality, since in a Boolean model with particles lying in parallel hyperplanes almost surely no overlaps occur. Thus, in this case the intensity of $X$ can be estimated directly since the Boolean model $Z$ and the particle process $X$ contain the same information. In particular, Theorem \ref{1} is trivially true in this case.

Under the given assumptions, the additivity properties of the functionals $V_j$ and $V(\cdot \, [j],K[d-j])$
ensure that also the deterministic densities $\overline V_j(Z)$ and $\overline V(Z [j],K[d-j])$ of the Boolean model $Z$ can be defined. One possibility would be to use the relations to the mean values of $X$, as they were described in the Introduction, relations which are specific for the current setting that is based on an underlying Poisson process. However, such densities can be defined for functionals $\varphi$ of more general random sets $Z$, either through a limit for a sequence of growing windows (which requires additivity of $\varphi$) or as a Radon-Nikodym derivative, which is based on local versions of the functionals $\varphi$. We refer to \cite[Chapter 9]{SW}, for details.

\section{The 4-dimensional case}

We complete now the proof of the case $d=4$ of Theorem \ref{1}. It motivates the ideas for the general case in the next section, but is simpler and therefore shorter due to the structure of the mixed mean values which occur. We  first recall the density formulas for the mixed volumes in four dimensions. They read
\begin{align}
\overline V_4(Z) &=1-\e^{-\overline V_4(X)},\label{4-4}\\
\overline V(Z[3],K[1]) &=\e^{-\overline V_4(X)}\overline V(X[3],K[1]),\label{4-3}\\
\overline V(Z[2],K[2]) &=\e^{-\overline V_4(X)}\Bigl(\overline V(X[2],K[2]) -\frac{1}{12}\overline V_{3,3,2}(X,X,K^*)\Bigr),\label{4-2}\\
\overline V(Z[1],K[3]) &=\e^{-\overline V_4(X)}\Bigl(\overline V(X[1],K[3]) -\frac{1}{4}\overline V_{2,3,3}(X,X,K^*)\nonumber\\
&\quad\quad  + \frac{1}{24}\overline V_{3,3,3,3}(X,X,X,K^*)\Bigr),\label{4-1}\\
\overline V_0(Z) &=\e^{-\overline V_4(X)}\Bigl( \overline V_0(X) -\overline V_{1,3}(X,X) -\frac{1}{2}\overline V_{2,2}(X,X)\nonumber\\ &\quad\quad  + \frac{1}{2}\overline V_{2,3,3}(X,X,X) -\frac{1}{24}\overline V_{3,3,3,3}(X,X,X,X)\Bigr).\label{4-0}
\end{align}
This is the corrected version of \cite[(20)]{W01}, both with respect to the missing term $\overline V_{2,2}(X,X)$ and to erroneous constants.
%(in \cite{W01}, a factor $\frac{1}{6}$ is also missing in the formula \cite[(18)]{W01} for the mean Euler characteristic in three dimensions).
We assume that all five mean values on the left are known for all $K$. Clearly, \eqref{4-4} determines the factor $\e^{-\overline V_4(X)}$ in all other equations. From \eqref{d-1-case} we thus obtain that \eqref{4-3} determines the mean area measure $\overline S_3(X,\cdot)$ and thus the Blaschke body $B(X)$. In \cite[Proposition 3]{W01}, it was shown that in all mean values, where the homogeneity index 3 occurs, the corresponding variable $X$ can be replaced by $B(X)$. For example,
$$\overline V_{3,3,2}(X,X,K^*) =  V_{3,3,2}(B(X),B(X),K^*)$$ which is determined, since $B(X)$ is known. Thus \eqref{4-2} yields $$\overline V(X[2],K[2]) = \tfrac{1}{6}\overline V_{2,2}(X,K^*)$$ for all $K$. Using the translative integral formula for the variable $K$, we get the mean mixed functional $\overline V_{2,3,3}(X,K,M)$ for all convex bodies $K,M$.
In fact, for all convex bodies $L,K,M\in\mathcal{K}^4_0$ and $\alpha,\beta\ge 0$ we have
\begin{align*}
&\int_{\R^4} V_{2,2}(L,(\alpha K)\cap (\beta M+x))\, \lambda_d(dx)\\
&\quad=\alpha^2\beta^4 V_{2,2}(L,K)V_4(M)+\alpha^4\beta^2 V_{2,2}(L,M)V_4(K)
+\alpha^3\beta^3V_{2,3,3}(L,K,M),
\end{align*}
and hence
\begin{align*}
&\int_{\R^4}  \overline{V}_{2,2}(X,(\alpha K)\cap (\beta M+x))\, \lambda_d(dx)\\
&\quad=\alpha^2\beta^4 \overline{V}_{2,2}(X,K)V_4(M) +
\alpha^4\beta^2 \overline{V}_{2,2}(X,M)V_4(K)
+\alpha^3\beta^3\overline{V}_{2,3,3}(X,K,M).
\end{align*}
In particular, we obtain the quantities
$$\overline V_{2,3,3}(X,B(X),K^*)= \overline V_{2,3,3}(X,X,K^*)$$
 in \eqref{4-1} and
$$\overline V_{2,3,3}(X,B(X),B(X))= V_{2,3,3}(X,X,X)$$
in \eqref{4-0}.
Thus, from \eqref{4-1}, $\overline V(X[1],K[3])$ is determined for all $K$. Again by \eqref{d-1-case}, this gives us $\overline h(X,\cdot)$  and hence
$$
\overline V_{1,3}(X,X) = \frac{1}{4}\int_{S^{d-1}} \overline h(X,u)\,\overline S_3(X,du) .
$$
Therefore, the only summand in \eqref{4-0} which remains to be discussed (apart from $\gamma =  \overline V_0(X)$) is $\overline V_{2,2}(X,X)$. As we will show, this mean value is determined by the knowledge of $\overline V(X[2],K[2]) $ for all $K$, which we have already obtained.

For a continuous function $g$ on the flag space $F(4,3)$, we consider the functional
$$
\varphi_g (M) = \int_{F(4,3)} g(u,L)\,\psi_2(M,d(u,L)),\quad M\in {\cal K}^d,$$
and its mean value
$$
\overline\varphi_g (X) = \gamma\int_{{\cal K}^d_0} \varphi_g (M)\, {\mathbb Q}(dM) =
\int_{F(4,3)} g(u,L)\, \overline\psi_2(X,d(u,L)).
$$
Obviously, $\varphi_g$ is an element of ${\bf Val}_2^{(4)}$, the vector space of translation invariant, continuous
valuations on $\R^4$ which are positively homogeneous of degree $2$.
By Alesker's result \cite{A1}, $\varphi_g$ can be approximated by linear combinations of special mixed volumes, that
is, for each $\varepsilon>0$ there are $n\in\N$, $\lambda_1,\ldots,\lambda_n\in\R$ and $K_1,\ldots,K_n\in\mathcal{K}^4_0$
such that
$$
|\varphi_g(M)-\sum_{i=1}^n\lambda_i V(M[2], K_i[2])|\le\varepsilon R(M)^2\le \varepsilon \cdot c_1 V_1(M)^2,
$$
for all $M\in\mathcal{K}^d$,
where $R(M)$ denotes the circumradius of $M$ and $V_1(M)\ge c_2\text{diam}(M)\ge c_2 R(M)$,
with some positive constant $c_1,c_2$, by the monotonicity of the intrinsic volume $V_1$.
Hence
$$
|\overline \varphi_g(X)-\sum_{i=1}^n\lambda_i \overline V(X[2], K_i[2])|\le \varepsilon \cdot c_3,
$$
with some constant $c_3>0$,
due to the assumption \eqref{moment} for $d=4$.
This shows that $\overline \varphi_g(X)$ is determined by $\overline V(X[2], K [2])$, for all $K\in\mathcal{K}^4_0$.
Notice that the mean values $\overline \varphi_g(X)$ exist, since
$$
 \int_{F(4,3)} |g(u,L)|\,\overline\psi_2(X,d(u,L))  \le \|g\| \overline V_2(X)<\infty .
$$
Thus we obtain $\int_{F(4,3)} g(u,L)\, \overline\psi_2(X,d(u,L))$ for all continuous functions $g$ and therefore $\overline\psi_2(X,\cdot)$.
But then also $\overline\psi_2(X^*,\cdot)$ is determined.

Let again $\varepsilon >0$ be arbitrary. In \cite{HRW}, an $\varepsilon$-approximation $V^{(\varepsilon)}(M[2],K^*[2])\ge 0$
of the mixed volume $V (M[2],K^*[2])$, for any $M,K\in\mathcal{K}^4_0$, has been introduced such that

$$
V^{(\varepsilon)}(M[2],K^*[2])\nearrow V (M[2],K^*[2]) ,
$$
as $\varepsilon\searrow 0$.
Moreover, it was shown that there exist bounded, measurable functions $f_2^{(\varepsilon)}$ on $F(4,3)\times F(4,3)$
such that
\begin{align*}
&V^{(\varepsilon)}(M[2],K^*[2])\\
&\qquad =\int_{F(4,3)}\int_{F(4,3)} f_2^{(\varepsilon)}
(u,L,u',L')\,\psi_2(K^*,d(u',L'))\,\psi_2(M,d(u,L)).
\end{align*}
Integration against $\gamma^2\mathbb{Q}^2$ then yields that
\begin{align}
&\int_{F(4,3)}\int_{F(4,3)} f_2^{(\varepsilon)}
(u,L,u',L')\,\overline\psi_2(X^*,d(u',L'))\,\overline\psi_2(X,d(u,L))\nonumber\\
&=\gamma^2 \int_{\mathcal{K}^4_0}\int_{\mathcal{K}^4_0}V^{(\varepsilon)}(M[2],K^*[2])\, \mathbb{Q}(dK)\, \mathbb{Q}(dM)\nonumber\\
&\nearrow \gamma^2 \int_{\mathcal{K}^4_0}\int_{\mathcal{K}^4_0}V (M[2],K^*[2])\, \mathbb{Q}(dK)\, \mathbb{Q}(dM)\nonumber\\
&=\overline V (X[2],X^*[2]),\label{intmean}
\end{align}
as $\varepsilon\searrow 0$. Since $\overline\psi_2(X,\cdot)$ and $\overline\psi_2(X^*,\cdot)$ are determined, we conclude that $6\overline V(X[2],X^*[2])=V_{2,2}(X,X)$ is determined as well. This completes the proof of the 4-dimensional case of Theorem \ref{1} and replaces the incomplete proof in \cite{W01}.

\section{Proof of the general case}

For general dimension $d\ge 5$, more complicated mixed functionals and their mean values occur, which cannot be expressed by the Blaschke body and the mean body anymore. Here, we need the recent extension \cite{HRW2} of the flag formulas. Moreover, the truncation argument which we used already in the previous section is more subtle and will be explained in detail.

To start with the proof of Theorem \ref{1}, we use again that $ \overline V_d(Z)$ and the equation
$$  \overline V_d(Z) = 1-\e^{-\overline V_d(X)}$$
determine the factor
$$ q= \e^{-\overline V_d(X)}$$
in the remaining equations. Thus
$$
 \overline V(Z [d-1], K[1]) = q\,\overline V(X [d-1],K [1])
$$
determines $\overline V(X [d-1],K [1])$, for all $K\in{\cal K}^d$. From \eqref{d-1-case} we get that this determines the mean area measure $\overline S_{d-1}(X,\cdot)$. Up to a constant, this is the mean flag measure $\overline \psi_{d-1}(X,\cdot)$ (if we identify $(u,\R^d)$ and $u$). We now proceed by recursion.

\begin{prop}\label{prop} Let $j\in\{ 1,\dots ,d-2\}$ and assume that $q$ and the mean flag measures $\overline\psi_{d-1}(X,\cdot),\dots,\overline\psi_{j+1}(X,\cdot) $ are given. Then the values of the densities $\overline V(Z [j], K[d-j] )$, for all $K\in{\cal K}^d$,  determine the mean flag measure $\overline \psi_j(X,\cdot)$.
\end{prop}

\begin{proof} We recall  formula \eqref{nonisotropic} in more explicit form
\begin{align}\label{nonisotropic2}
  \binom{d}{j}\overline V(Z [j],& K[d-j] ) = q \Biggl(\binom{d}{j}\overline V(X [j], K[d-j] )\\
  &- \sum_{k=2}^{d-j} \frac{(-1)^k}{k!} \sum_{\substack{m_1,\ldots ,m_k=j+1\\ m_1+\ldots + m_k=(k-1)d+j}}^{d-1}
  \overline V_{m_1,\dots,m_k,d-j}(X,\dots,X, K^*)\Biggr).\nonumber
\end{align}
%Here, we used the connection
%$$V_{m_1,\dots,m_k}(K_1,\dots ,K_k,K[d-j] ) = V_{m_1,\dots,m_k,d-j}(K_1,\dots ,K_k,K^*)$$
%and omitted the upper index $(0)$ on the right-hand side, as pointed out before.

In Hug-Rataj-Weil \cite[Section 6]{HRW2}, the formula
\begin{align}\label{genflagform}
&V_{\bf m}(K_1,\dots ,K_k,K^*)\nonumber\\
&\ = \int_{F(d,d-j+1)}\int_{F(d,m_k+1)}\cdots \int_{F(d,m_1+1)}f_{\bf m}(u_1,L_1,\dots ,u_k,L_k,u,L)\nonumber\\
&\qquad \times \,\psi_{m_1}(K_1,d(u_1,L_1))\cdots \psi_{m_k}(K_k,d(u_1,L_1))\, \psi_{d-j}(K^*,d(u,L))
\end{align}
was proved. Here, the index ${\bf m}$ stands shortly for $(m_1,\dots ,m_k,d-j)$ and the function $f_{\bf m}$ is of the form
$$
f_{\bf m}(u_1,L_1,\dots ,u_k,L_k,u,L) = G_{\bf m}(u_1,\dots ,u_k,u)\cdot \phi_{\bf m}(u_1,L_1,\dots ,u_k,L_k,u,L).
$$
Also, $G_{\bf m}\ge 0$ is measurable, but unbounded, whereas $\phi_{\bf m}$ is a signed continuous function. Both functions are given explicitly in \cite{HRW2}. The special case where $k=1$ was treated in \cite{HRW} and has been used
in the previous section for dimension $d=4$.
In this generality, the formula requires that the bodies $K_1,\dots ,K_k,K^*$ are in general position. We need not go into details here, since we will use an argument to avoid this condition. Namely we consider, as in \cite{HRW2}, the bounded approximation
$$V_{\bf m}^{(\varepsilon)}(K_1,\dots ,K_k,K^*)\ge 0,\quad \varepsilon > 0,$$
which can be described similarly as in \eqref{genflagform}, but with a function $f_{\bf m}^{(\varepsilon)}$ which is obtained by replacing $G_{\bf m}(u_1,\dots ,u_k,u)$   by the bounded, measurable and non-negative function
$$
G_{\bf m}^{(\varepsilon)}(u_1,\dots ,u_k,u) = G_{\bf m}(u_1,\dots ,u_k,u){\bf 1}\{\operatorname{dist}(o,\text{conv}\{
u_1,\dots
,u_k,u\})\ge \varepsilon\} ,$$
where $\text{conv}\{
u_1,\dots
,u_k,u\}$ denotes the convex hull of $u_1,\dots
,u_k,u\in\R^d$. Since $G_{\bf m}(u_1,\dots ,u_k,u)=0$ if $u_1,\dots ,u_k,u$ are linearly dependent, we obtain
$$
G_{\bf m}^{(\varepsilon)} \nearrow G_{\bf m}$$
as $\varepsilon \searrow  0$. As is shown in \cite[Section 6]{HRW2}, this implies
$$
V_{\bf m}^{(\varepsilon)}(K_1,\dots ,K_k,K^*)\nearrow V_{\bf m}(K_1,\dots ,K_k,K^*)
$$
for all $K_1,\dots ,K_k,K \in{\cal K}^d$, without an additional condition of general position. The reason is that $V_{\bf m}(K_1,\dots ,K_k,K^*)$ obeys a curvature representation over the normal bundles of the bodies $K_1,\dots ,K_k,K^*$ (this was shown in \cite{HR}) and the truncated functional $V_{\bf m}^{(\varepsilon)}(K_1,\dots ,K_k,K^*)$ has a similar representation (this was shown in \cite{HRW2}).

Now we can proceed with the proof.
The definition of the truncated functionals and the monotone convergence transfers to the mean values
\begin{align*}
&\overline V_{\bf m}^{(\varepsilon)}(X,\dots ,X,K^*)\\
&\ = \int_{F(d,d-j+1)}\int_{F(d,m_k+1)}\cdots \int_{F(d,m_1+1)}f_{\bf m}^{(\varepsilon)}(u_1,L_1,\dots ,u_k,L_k,u,L)\nonumber\\
&\qquad \times \overline\psi_{m_1}(X,d(u_1,L_1))\cdots \overline\psi_{m_k}(X,d(u_1,L_1))\, \psi_{d-j}(K^*,d(u,L))
\end{align*}
and
\begin{align*}
&\overline V_{\bf m}^{(\varepsilon)}(X,\dots ,X,K^*)\\
&\qquad =\gamma^k \int   V_{\bf m}^{(\varepsilon)}(K_1,\dots ,K_k,K^*)\, \mathbb{Q}^k(d(K_1,\ldots,K_k))\\
&\qquad \nearrow \gamma^k \int   V_{\bf m} (K_1,\dots ,K_k,K^*)\, \mathbb{Q}^k(d(K_1,\ldots,K_k))\\
&\qquad =
\overline V_{\bf m}(X,\dots ,X,K^*),
\end{align*}
as $\varepsilon \searrow 0$, compare the derivation of \eqref{intmean}.
Due to our assumptions, the mean flag measures $\overline\psi_{m_1}(X,\cdot),\dots,$
$\overline\psi_{m_k}(X,\cdot)$ are determined (recall that $m_i\ge j+1$), and thus $\overline V_{\bf m}^{(\varepsilon)}(X,\dots ,X,K^*)$ is determined for all $\varepsilon>0$, which gives us $\overline V_{\bf m}(X,\dots ,X,K^*)$ for all $K$. From \eqref{nonisotropic2}, we deduce that $\overline V(X [j], K[d-j] )$ is given for all $K$.

In the remaining part of the proof, we argue as in the 4-dimensional case. Let $\varphi_g$ be a valuation of the form
$$
\varphi_g (M) = \int_{F(d,j+1)} g(u,L)\,\psi_j(M,d(u,L)),
$$
where $M\in\mathcal{K}^d$ and $g$ is a continuous function on $F(d,j+1)$. Any such valuation is translation invariant,
continuous and homogeneous
of degree $j$. From Alesker's result we obtain that for each $\varepsilon>0$ there are $n\in\N$,
$\lambda_1,\ldots,\lambda_n\in\R$ and $K_1,\ldots,K_n\in\mathcal{K}^d$ such that
$$
|\varphi_g(M)-\sum_{i=1}^n\lambda_i V(M[j],K_i [d-j])|\le c_4\, \varepsilon V_1(M)^j,
$$
with a constant $c_4$ (depending only on the dimension $d$),
for all $M\in\mathcal{K}^d$. Hence, under the moment assumption \eqref{moment},
we conclude that
 the mean values $\overline\varphi_g (X)$ are determined.
The corresponding integrals then determine $\overline\psi_j(X,\cdot )$.
\end{proof}

To complete the proof of Theorem \ref{1}, we notice that we obtained already the mean values $p= \e^{-\overline V_d(X)}$ and $\overline \psi_{d-1}(X,\cdot)$, the starting point for the recursion. Proposition \ref{prop} thus shows that all mean flag measures
$$\overline \psi_{d-1}(X,\cdot),\dots , \overline \psi_{1}(X,\cdot)$$
are determined.  Using \eqref{genflagform} again,
together with the truncation method explained above, we obtain all mean values
$$
\overline V_{m_1,\dots, m_k}(X,\dots,X)
$$
with $m_i\in\{ 1, \dots ,d-1\}$ and such that $m_1+\ldots +m_k= (k-1)d$.

By \eqref{Euler}, the intensity $\gamma$ is determined.

\section{Remarks}

In the following we comment on extensions and applications of the main result, Theorem \ref{1}.

\medskip\noindent
{\bf 1}. All notions used in the theorem and its proof, like mixed volumes, mixed translative functionals and flag measures, are additive and continuous (scalar or measure-valued) functionals  on ${\cal K}^d$. Hence they have an additive extension to {\it polyconvex} sets (finite unions of convex bodies). Therefore, Theorem \ref{1} also holds for Boolean models $Z$ with polyconvex grains, after appropriate modifications. We mention that the results in \cite{W01} were obtained for polyconvex grains. This required a modified integrability condition (see \cite[formula (2)]{W01}) and correspondingly \eqref{moment} has to be adjusted. Also, the assumption has to be made that the particles have Euler characteristic one, since otherwise we do not obtain the intensity $\gamma$ but $\gamma \overline V_0(X)$. We refrain from formulating the corresponding result.

\medskip\noindent
{\bf 2}. Let $\bf Val$ denote the space of translation invariant continuous valuations on ${\cal K}^d$ and ${\bf Val}_j$ the subspace of valuations which are homogeneous of degree $j$, $j=0,\ldots,d$. Every $\varphi\in {\bf Val}$ has a unique decomposition $\varphi = \sum_{j=0}^d \varphi_j$, $\varphi_j\in {\bf Val}_j$, where $\varphi_0$ is a multiple of $V_0$ and $\varphi_d$ is a multiple of $V_d$. Alesker's approximation result indicates that the formulas \eqref{nonisotropic} have counterparts for the homogeneous components $\varphi_j$, $j=0,\ldots,d$, of $\varphi\in {\bf Val}$, where mean values of $X$ for certain mixed valuations occur on the right. This is indeed the case and was obtained in \cite[Corollary 6.3]{W17} in a direct way using translative integral formulas for the functionals $\varphi_j$.

\medskip\noindent
{\bf 3.} In this context, it should be remarked that the mean values for mixed volumes of $Z$ not only determine the mean flag measures
$$
\int_{{\cal K}^d_0} \psi_j(K,\cdot)\, {\mathbb Q} (dK),\quad j=1,\dots ,d-1,
$$
(as formulated in Theorem \ref{1}), but also all mean values
$$
\int_{{\cal K}^d_0} \varphi_j(K)\, {\mathbb Q} (dK),$$
for $\varphi_j\in {\bf Val}_j$, $j=0,\dots ,d$ (as it was demonstrated in the proof). For isotropic $Z$, this does not lead to new information, since then $\varphi_j$ can be assumed to be rotation invariant. By Hadwiger's theorem (see \cite[Theorem 6.4.3]{S}), this gives $\overline V_j(X)$, thus we are back in the situation of \eqref{MilesDavy}, which is clear since \eqref{MilesDavy} and \eqref{nonisotropic} are equivalent in the isotropic case.

\medskip\noindent
{\bf 4.} For non-isotropic $Z$, the knowledge of $\overline\varphi_j (X)$ can give us some information on the shape. For $j=1$, we can choose $\varphi_1(K) = h(K,\cdot)$. Thus we obtain the Minkowski mean $\tilde M(X):= \gamma^{-1}M(X)$ of $X$. This already shows that the underlying Poisson process is completely determined, if the distribution $\mathbb Q$ is concentrated on a single shape. As images of the mean flag measures, we also get the mean area measures
$$
\int_{{\cal K}^d_0} S_j(K,\cdot)\, {\mathbb Q} (dK),\quad j=1,\dots ,d-1.
$$
In particular, we get the Blaschke mean $\tilde B(X):=\gamma^{-1}B(X)$. This, together with $\tilde M(X)$ and the mean particle volume $\tilde V_d(X) :=\gamma^{-1}\overline V_d(X)$, determines the process $X$, if $\mathbb Q$ is concentrated on two different shapes $K_1$ and $K_2=cK_1, c\not= 1$, with inner points. In general, the mean area measures give us the first $d$ moments of the distribution ${\mathbb P}_\xi$, if the typical grain is of the form $\xi K_0$, with a fixed body $K_0\in{\cal K}^d_0$ and a random variable $\xi\ge 0$. For many parametric situations this implies that the whole distribution ${\mathbb Q}$ (and thus the body $K_0$) is determined.

\section{Non-stationary Boolean models}

In \cite{W01} (see also \cite[Section 11.1]{SW}), the formulas \eqref{nonisotropic} were generalized to certain non-stationary Boolean models $Z$, namely those for which the underlying Poisson process $X$ has a translation regular intensity measure $\Theta$. This condition means that $\Theta$ is absolutely continuous with respect to a translation invariant, locally finite measure $\tilde\Theta$ on ${\cal K}^d$. Explicitly, we have
\begin{equation}\label{transregular}
\Theta (A) = \int_{{\cal K}^d_0}\int_{\R^d} {\bf 1}_A(K+x)\eta(K,x)\,\lambda_d(dx)\,{\Q}(dK),
\end{equation}
for all Borel sets $A\subset {\cal K}^d$ and with a measurable function $\eta\ge 0$ on ${\cal K}^d_0\times\R^d$. In order to simplify the following presentation, we make the additional assumption that $\eta$ is continuous and does not depend on $K$, that is, $\eta (x)$ is the spatial intensity of $X$ at $x\in\R^d$. We call such a Boolean model {\it regular}. Under these conditions, the decomposition \eqref{transregular} is unique and the densities of mixed volumes and mixed functionals in the following formulas exist pointwise (not only almost everywhere). The counterpart to \eqref{nonisotropic} then reads
\begin{align}
  \overline V_d(Z;z) &= 1-\e^{-\overline V_d(X;z)} ,\label{vol}
  \\
  \overline V(Z [d-1], K[1];z) &= \e^{-\overline
    V_d(X;z)}\overline V(X [d-1],K [1];z) ,\label{surf}
\end{align}
and
\begin{align}\label{nonstationary}
  &\binom{d}{j}\overline V(Z [j], K[d-j];z ) \nonumber\\
	&\qquad =\e^{-\overline V_d(X;z)}
  \sum_{\mathbf{m}\in\operatorname{mix}(j)}
  \frac{(-1)^{{|\mathbf{m}|}-1}}{|\mathbf{m}|!}\overline V_{\mathbf{m},d-j}(X,\dots,X, K^*;z,\dots ,z),
\end{align}
for $j=0,\dots ,d-2$, all $z\in\R^d$, and all $K\in{\cal K}^d$. Here, on the left of \eqref{vol}, the function $\overline V_d(Z;\cdot)$ is defined as the Radon-Nikodym derivative of the expectation measure $\E \lambda_d(Z,\cdot)$ (where $\lambda_d(Z,\cdot)$ denotes the restriction of $\lambda_d$ to $Z$). Similarly, $\binom{d}{j}\overline V(Z [j], K[d-j];\cdot )$, for $j=0,\dots, d-1,$ is defined as the Radon-Nikodym derivative of the expectation measure $\E \Phi_{j,d-j}(Z,K^*;\cdot\times \R^d)$, where $\Phi_{j,d-j}(Z,K^*;\cdot)=\Phi_{j,d-j}^{(0)}(Z,K^*;\cdot)$ is the mixed translative measure, the local variant of the mixed functional $V_{j,d-j}(Z,K^*)$ (see \cite[Section 5.2]{SW}). Recall here that $\binom{d}{j}V(M[j],K[d-j]) =V_{j,d-j}(M,K^*)$.  Also, it should be mentioned that $\Phi_{j,d-j}(Z(\omega),K^*;\cdot\times\R^d)$ exists by additive extension and is a locally finite signed measure on $\R^d$, for all sets $Z(\omega)$ in the extended convex ring. Thus, the random variable $\Phi_{j,d-j}(Z,K^*;A\times\R^d)$ is defined and finite for every bounded Borel set $A\subset\R^d$.

On the right, the mean mixed functionals $\overline V_{\mathbf{m},d-j}(X,\dots,X, K^*;\cdot)$ are defined more directly as Radon-Nikodym derivatives of the expectation measures
\begin{equation}\label{instatexp}
\E\sum_{(M_1,\ldots ,M_j)\in X^j_{\not =}}\Phi_{\mathbf{m},d-j}(M_1,\ldots,M_j,K^*;\cdot\times\R^d)
\end{equation}
with respect to $\lambda_d^j$ (the $j$-fold product measure of $\lambda_d$). The  summation here is over all $j$-tuples $(M_1,\dots, M_j)\in (\mathcal{K}^d)^j$ of pairwise different particles $M_i$ in $X$ and the mixed measure $\Phi_{\mathbf{m},d-j}(M_1,\ldots,M_j,K^*;\cdot)$
 is a finite measure on $(\R^d)^{j+1}$ and arises from the (iterated) translative integral formula for the local analog of the intrinsic volumes, the curvature measures (see \cite[Section 5]{W01} and \cite[Section 6.4]{SW}). As was shown in \cite[Theorem 5]{W01}, the measure \eqref{instatexp} is indeed locally finite and absolutely continuous. Moreover, the derivative satisfies
\begin{align}\label{derivative}
&\overline V_{\mathbf{m},d-j}(X,\dots,X, K^*;z_1,\ldots, z_j) \nonumber\\
&\qquad = \int_{({\cal K}_0^d)^j}\int_{(\R^d)^j}\eta(z_1-x_1)\cdots\eta(z_j-x_j)\\
&\qquad  \qquad \times\Phi_{\mathbf{m},d-j}(M_1,\ldots,M_j,K^*;d(x_1,\ldots ,x_j)\times{\R}^d)\,\Q^j(d(M_1,\ldots ,M_j)) .\nonumber
\end{align}
Since we require $\eta$ to be continuous, this integral representation holds for all $(z_1,\dots ,z_j)\in(\R^{d})^j$, not only almost everywhere.
The absolute continuity of $\overline V(Z [j], K[d-j];\cdot\times\R^d )$ was then shown in \cite[Theorem 6]{W01} (see also \cite[Theorem 11.1.2]{SW}). Here, in lines 10--13 of the proof \cite[p.~54]{W01},  the differential $\lambda_d(dz_1)\cdots \lambda_d(dz_k)$ has to be replaced by $\lambda_d(dz)$ and consequently the variables $z_1,\dots, z_k$ in the integrands have to be replaced by one variable $z$.

Since for $j=0$ the decomposition result, which we used for  mixed volumes and mixed functionals, also holds for the mixed measures in \eqref{derivative}, we obtain from \eqref{nonstationary} the following formula for the density of the Euler characteristic at $z\in\R^d$,
\begin{align}\label{nonstationaryEuler}
  \overline V_0(Z;z )  &=\e^{-\overline V_d(X;z)}
  \sum_{\mathbf{m}\in\operatorname{mix}(0)}
  \frac{(-1)^{{|\mathbf{m}|}-1}}{|\mathbf{m}|!}\overline V_{\mathbf{m}}(X,\dots,X;z,\dots ,z).
\end{align}

In order to discuss the question which information on $X$ can be extracted from the local densities of $Z$ in \eqref{nonstationary}, let us start with the simpler equations \eqref{vol} and \eqref{surf}. Obviously, \eqref{vol}
determines $q(z):= \e^{-\overline V_d(X;z)}$ and, given $q(z)$, \eqref{surf}
determines $\overline V(X [d-1],K [1];z)$, for all $K\in {\cal K}^d$. From \eqref{derivative} we get
\begin{align*}
d\overline V(X [d-1],K [1];z) =  \int_{{\cal K}_0^d}\int_{\R^d}\eta(z-x)\,\Phi_{d-1,1}(M,K^*;dx\times{\R}^d)\,\Q(dM).
\end{align*}
and \cite[Corollary 3]{GW02} yields
\begin{align*}
d\overline V(X [d-1],K [1];z) &=  \int_{{\cal K}_0^d}\int_{\Nor (M)}\eta(z-x)h^*(K,u)\,\Theta_{d-1}(M;d(x,u))\,\Q(dM)\\
&= \int_{S^{d-1}} h^*(K,u)\, \mu_z(du) ,
\end{align*}
where $\mu_z$ is the measure
$$
\mu_z (\cdot):= \int_{{\cal K}_0^d}\int_{\Nor (M)}\eta(z-x){\bf 1}(u\in\cdot)\,\Theta_{d-1}(M,d(x,u))\,\Q(dM)
$$
on $S^{d-1}$ and $\Theta_{d-1}(M,\cdot)$ is the $(d-1)$-st support measure (normalized as in \cite{S}) on the normal bundle $\Nor (M)$ of $M$.
If we write $\omega_{d-1}$ for spherical Lebesgue measure and center $\mu_z$ by
$$
\mu_z^* (\cdot):= \mu_z(\cdot) - \frac{1}{\kappa_d}\int_{S^{d-1}} {\bf 1}(u\in\cdot)\langle u,c\rangle\,\omega_{d-1}(du),
$$
with
$$c:= \int_{S^{d-1}} u \,\mu_z(du),$$
then the collection of densities $\overline V(X [d-1],K [1];z), K\in{\cal K}^d$, determines $\mu_z^*$ uniquely (see \cite[Lemma 4]{GW02}).

This observation already emphasizes a difficulty which arises for non-statio\-nary $Z$. Namely, mixed volumes are translation invariant in each component. Therefore, they only provide information on the bodies involved which is of `centred type' (surface area measures, flag measures, centred support functions, etc.). Consequently, in the non-stationary case, a centredness property of the Boolean model $Z$ seems necessary. Let us assume, for fixed $z\in\R^d$,  that $Z$ is {\it centred at $z$}, meaning that $c=0$  (that is, $\mu_z$ is centred) or, equivalently, $\mu_z^* = \mu_z$.

In dimensions 2 and 3, and for a Boolean model $Z$ which is centred at $z$, we obtain uniqueness theorems similar to our Theorem \ref{1}. The case $d=2$ was discussed in \cite{W00} (and again in \cite{GW02}) and the case $d=3$ in \cite{GW02}. It was shown that the local densities $\overline V(Z [j],K [d-1];z)$, for $j=0,\dots ,d$ and all $K\in {\cal K}^d$, determine the convoluted intensity function
$$
\overline V_0(X;z) := \int_{\R^d} \eta (z-x)\,\tilde\Phi_0(X,dx),
$$
where
$$
\tilde\Phi_0(X,\cdot) := \int_{{\cal K}_0^d} \,\Phi_0(M,\cdot)\,\Q(dM)
$$
is the expected (with respect to $\Q$) Gaussian curvature measure of the particles. If $Z$ is stationary, then $\eta\equiv \gamma$, hence $\overline V_0(X;z)= \gamma$ since $\Phi_0(M,\R^d) = V_0(M)= 1$. Note that these results are of local type, they hold for a fixed position $z\in\R^d$ and require only the knowledge of the local densities of $Z$ in $z$. Since in \cite{W00} and \cite{GW02} there were some erroneous statements about the centredness of certain functions, we sketch here the proofs of these results. As explained above, we obtain $q(z)$ and $\mu_z^*$. Then, for $d=2$, only the equation \eqref{nonstationaryEuler} remains, reading
\begin{align*}
  \overline V_0(Z;z )  &=q(z)\left(\overline V_0(X;z) - \frac{1}{2}\overline V_{1,1}(X,X;z,z)\right).
\end{align*}
As was shown in \cite{GW02} (combine p.~332, lines -1 to -3, and Corollary 2),
$$
\overline V_{1,1}(X,X;z,z) = \int_{S^1}\int_{S^1} \alpha (u,v) \det (u,v)\mu_z(du)\mu_z(du),
$$
where $\alpha (u,v)\in [0,\pi]$ is the angle between $u$ and $v$ and $\det (u,v)$ is the absolute determinant. For centred $Z$,
$\mu_z=\mu_z^*$ is known, which gives us $\overline V_{1,1}(X,X;z,z)$ and thus $\overline V_0(X;z)$.

For $d=3$, we have two remaining equations,
$$
 \overline V(Z[1],K[2];z )  =q(z)\frac{1}{3}\left(\overline V_{1,2}(X,K^*;z) - \frac{1}{2}\overline V_{2,2,2}(X,X,K^*;z,z)\right)
$$
and
$$
 \overline V_0(Z;z )  =q(z)\left(\overline V_0(X;z) - \overline V_{1,2}(X,X;z,z)+\frac{1}{6}\overline V_{2,2,2}(X,X,X;z,z,z)\right).
$$
Concerning the first equation, it was shown in \cite[p.~342]{GW02} that
\begin{align*}
\overline V_{2,2,2}&(X,X,K^*;z,z)\\
&= \int_{S^2}\int_{S^2} \int_{S^2}\alpha (u,v,w) \det (u,v,w)\mu_z(du)\mu_z(dv)S_2(K^*,dw),
\end{align*}
where $\alpha (u,v,w)$ is now a spherical angle between $u,v,w$ (the spherical volume of the spherical triangle spanned by these vectors) and $\det (u,v,w)$ is the absolute determinant. Again, using that $Z$ is centred and hence that $\mu_z=\mu_z^*$ is known, we obtain $\overline V_{1,2}(X,K^*;z)$. Letting $K$ vary, Lemma 5 in \cite{GW02} shows that $h_z^*$ is determined for $d=3$. Here,
$$
h_z(u) := \int_{{\cal K}_0^d}\int_{\R^d}\eta(z-x)\rho(M;u,dx)\,\Q(dM),\quad u\in S^{d-1},
$$
is the density of the support function of $X$, defined via the local version of the support function $h(M,\cdot)$, the support kernel $\rho (M;\cdot,dx)$, and $h_z^*$ is the centred version
$$
h_z^*(\cdot) := h_z(\cdot) - \frac{1}{\kappa_d}\left\langle\cdot , \int_{S^{d-1}} uh_z(u)\,\omega_{d-1}(du)\right\rangle.
$$

From \cite[p.~343]{W01}, we get
$$
\overline V_{2,2,2}(X,X,X;z,z,z) = \int_{S^2}\int_{S^2} \int_{S^2}\alpha (u,v,w) \det (u,v,w)\mu_z(du)\mu_z(dv)\mu_z(dw).
$$
Using \cite[p.~344]{W01}, we obtain
\begin{align*}
\overline V_{1,2}(X,X;z,z) &= \int_{S^2} h_z(-u)\mu_z(du)\\
&= \int_{S^2} h_z^*(-u)\mu_z^*(dv),
\end{align*}
where we used that $\mu_z=\mu_z^*$ is centred (by assumption). This finally shows that
  $V_0(X;z)$ is determined.  We emphasize that the assumption $h_z=h_z^*$ is not necessary for this argument (and also not sufficient, in contrast to a claim in \cite{GW02}).

In these small dimensions 2 and 3, spherical quantities like the measure $\mu_z$ and the function $h_z$, as well as their centred versions are sufficient for integral representations of the mean mixed functionals which occur in \eqref{nonstationary}. If we look at the proof of Theorem \ref{1}, for $d\ge 4$, it combines two fundamental steps which would be essential also in the non-stationary case. The first is the use of flag measures to obtain a corresponding integral representation of the mean mixed functionals of $X$. This step can in fact be made also for the local mixed densities in \eqref{derivative}, using a local flag representation of the mixed curvature measures $\Phi_{\mathbf{m},d-j}(M_1,\ldots,M_j,K^*;\cdot)$ proved in \cite[Theorem 5]{HRW2} (and based on a recent result in \cite{HR}). Since \eqref{derivative} depends on the location $z\in \R^d$, a more general flag measure is involved, the {\it flag support measure} $\theta_j(K,\cdot)$. The latter is a finite Borel measure on the extended flag space $\R^d\times F(d,j+1)$ which, in extension of \eqref{flag}, is defined by
$$
\theta_j(K,\cdot)=\int_{G(d,j+1)}\int_{S^{d-1}\cap U}\mathbf{1}\{(x,u,U)\in\cdot\}\, \Theta_j^U(K|U,d(x,u))\, \nu_{j+1}(dU).
$$
Here, $\Theta_j^U(K|U,\cdot)$ is the $j$-th support measure of the projection of $K$ onto $U$, calculated in
the subspace $U$. For properties, different normalizations and isomorphic versions of support measures and their flag versions, we refer to \cite{S} and \cite{HTW}.

The second essential step in the proof of Theorem \ref{1} was to proceed from the mean mixed volumes $\overline V (X [j], K[d-j])$, for $K\in{\cal K}^d,$ to the mean flag measure $\overline\psi_j(X,\cdot)$, for the next round of the recursion. This step was based on Alesker's approximation of translation invariant, continuous valuations on ${\cal K}^d$ by linear combinations of mixed volumes. In the non-stationary case, in each recursion step, we  would obtain the local density
\begin{align*}
\overline V(X [j],K [d-j];z) =  \int_{{\cal K}_0^d}\int_{\R^d}\eta(z-x)\,\Phi_{j,d-j}(M,K^*;dx\times{\R}^d)\,\Q(dM)
\end{align*}
and would need to show that these densities, given for all $K\in{\cal K}^d$ (and a fixed $z\in\R^d$), determine the convoluted mean flag support measure
$$ \mu_{z,j}(\cdot):=\int_{{\cal K}^d_0}\int_{\R^d\times F(d,j+1)}\eta(z-x){\bf 1}((u,L)\in\cdot ) \,\theta_j(M,d(x,u,L))\,\Q(dM),$$
respectively its centred version $ \mu_{z,j}^*$.
Here, Alesker's theorem seems not applicable, since the dependence on $z$ (and thus the occurrence of $x$ in the flag measure) does not allow to define an appropriate valuation which is translation invariant.

It is also likely that further centredness conditions like $\mu_{z,j}=\mu_{z,j}^*$ are needed for such a step. We mention that a simple sufficient condition which implies that all densities of area measures, support functions and flag measures at $z$ are centred quantities is given by the symmetry of $\eta$ at $z$ (that is $\eta (z-x)=\eta (z+x)$ for all $x\in\R^d$) together with the symmetry of the particles (that is $K=K^*$, for $\Q$-almost all $K$).

\end{document}